\documentclass[12pt,a4paper]{amsart}
\usepackage[bottom]{footmisc}
\usepackage{mathrsfs}
\usepackage{amsmath}
\usepackage{amsthm}
\usepackage{amsfonts}
\usepackage{latexsym}
\usepackage{graphicx}
\usepackage{hyperref}
\usepackage{amssymb}
\usepackage{subfigure}
\usepackage{epsf}
\usepackage{float}
\usepackage{fancyhdr}
\allowdisplaybreaks \makeatletter
\def\rightharpoonfill@{\arrowfill@\relbar\relbar\rightharpoonup}
\DeclareRobustCommand{\overrightharpoon}{\mathpalette{\underarrow@\rightharpoonfill@}}
\makeatother

\begin{document}
\newcommand{\beq}{\begin{equation}}
\newcommand{\eneq}{\end{equation}}
\newtheorem{thm}{Theorem}[section]
\newtheorem{coro}[thm]{Corollary}
\newtheorem{lem}[thm]{Lemma}
\newtheorem{prop}[thm]{Proposition}
\newtheorem{defi}[thm]{Definition}
\newtheorem{rem}[thm]{Remark}
\newtheorem{cl}[thm]{Claim}
\title{An approximation method for the optimization of $p$-th moment of $\mathbb{R}^n$-valued random variable}
\author{Yanhua Wu$^1$\ \ \ \ \ \ \ Xiaojun Lu$^{2}$}
\thanks{Corresponding author: Xiaojun Lu, Department of Mathematics \& Jiangsu Key Laboratory of
Engineering Mechanics, Southeast University, 210096, Nanjing, China}
\thanks{Email addresses: yhwu85@126.com(Yanhua Wu), lvxiaojun1119@hotmail.de(Xiaojun Lu)}
\thanks{Keywords: optimal probability density, singular variational
problem, canonical duality theory}
\thanks{Mathematics Subject Classification(2010): 35J20, 35J60, 49K20,
60E05, 80A20}
\date{}
\maketitle
\pagestyle{fancy}                   
\lhead{Y. Wu and X. Lu}
\rhead{An approximation method for the optimization of $p$-th moment} 
\begin{center}
1. Department of Sociology, School of Public Administration, Hohai
University, 211189, Nanjing, China\\
2. Department of Mathematics \& Jiangsu Key Laboratory of
Engineering Mechanics, Southeast University, 210096, Nanjing, China

\end{center}
\date{}
\maketitle
\begin{abstract} This paper mainly addresses the optimization of $p$-th moment of $\mathbb{R}^n$-valued random variable. Through an ingenious approximation
mechanism, one transforms the maximization problem into a sequence
of minimization problems, which can be converted into a sequence of
nonlinear differential equations with constraints by variational
approach. The existence and uniqueness of the solution for each
equation can be demonstrated by applying the canonical duality
method. Moreover, the dual transformation gives a sequence of
perfect dual maximization problems. In the final analysis, one
constructs the approximation of the probability density function
accordingly.
\end{abstract}
\renewcommand{\abstractname}{R\'{e}sum\'{e}}
\begin{abstract}
Dans cet article, on consid\`{e}re essentiellement le probl\`{e}me
de maximisation du moment d'ordre $p$ pour $\mathbb{R}^n$-vecteur
al\'{e}atoire. En utilisant un m\'{e}canisme d'approximation
ing\'{e}nieuse, on transforme le probl\`{e}me en une s\'{e}quence de
probl\`{e}mes de minimisation, qui peut \^{e}tre convertie en une
s\'{e}quence des \'{e}quations diff\'{e}rentielles nonlin\'{e}aires
avec des contraintes par la m\'{e}thode variationelle. En
particulier, nous prouvons l'existence et l'unicit\'{e} de la
solution en appliquant la m\'{e}thode de dualit\'{e} canonique. De
plus, la transformation de dualit\'{e} donne une s\'{e}quence des
duals parfaits de maximisation. Enfin, nous \'{e}tudions
l'approximation de la densit\'{e} de probabilit\'{e}.
\end{abstract}

\section{Introduction}
Optimization and probability theory is frequently used in modern
financial studies, such as option pricing, portfolio investment,
asset management etc. In practice, for instance, a portfolio manager
of a mutual fund may invest in diversified securities to maximize
the returns from increases in the prices of the securities on hand.
Interested readers can refer to \cite{Co,W,Z} for more
details.\\

In this paper, we discuss a typical abstract model in these
respects, namely, maximization of the $p$-th moment of
$\mathbb{R}^n$-valued random variable. Let
$\Omega=\mathbb{B}(O,R_1)\backslash \mathbb{B}(O,R_2)$, $R_1>R_2>0$,
where $\mathbb{B}(O,R_1)$ and $\mathbb{B}(O,R_2)$ denote open balls
with center $O$ and radii $R_1$ and $R_2$ in the Euclidean space
$\mathbb{R}^n$, respectively. Let $(\Omega,F,\mathbb{P})$ be a
probability space. $X$ is real-valued random variable, then
\[\mathbb{E}(|X|^p)=\int_\Omega |X(\omega)|^p{\rm d} \mathbb{P}(\omega),\ (p>0)\] is said to be the $p$-th moment of $X$
for $X\in L^p(\Omega,\mathbb{R}^n)$, which denotes the family of
$\mathbb{R}^n$-valued random variable $X$ with
$\mathbb{E}(|X|^p)<\infty$. Readers can refer to \cite{Z} for more
details concerned with convergence and moment inequalities in this
weighted $L^p$-space. Let $\alpha>0$ be sufficiently large and
consider the radially symmetric probability densities subject to the
following constraints, \beq u\in W_0^{1,\infty}(\Omega)\cap
C(\overline{\Omega}), \eneq \beq u\geq 0,\ \text{a.e.\ in}\
\Omega,\eneq \beq\|u\|_{L^1(\Omega)}=1,\eneq \beq\|\nabla
u\|_{L^\infty(\Omega)}\leq\alpha,\eneq where
$W_0^{1,\infty}(\Omega)$ is the Sobolev space \cite{RA}. In this
paper, we focus on the following inverse problem, namely,
maximization of the $p$-th moment of the real-valued random variable
$Y\in\Omega$ with respect to the probability densities $u$ subject
to (1)-(4), \beq
(\mathcal{P}):\displaystyle\max_{{u}}\Big\{\mathbb{E}_{u}(|Y|^p)
:=\int_{\Omega}|y|^pu(y)dy\Big\}.\eneq

If $p\geq1$, then $|y|^p$ is a convex function on $\Omega$. In this
sense, it is natural to apply our analysis to general convex
functions, which represent typical payoff functions in the financial
system. And the maximization problem (5) is aimed to find an optimal
investment strategy in order to maximize the profit. Furthermore,
the constraint (1) $u=0$ on $\partial\Omega$ requires the manager to
invest in neither the most profitable but highly risky nor
unprofitable projects. While the constraint (4) establishes the
principle of a diversified investment portfolio, which means, it is
not allowed to put all eggs in one basket. Our discussion can also
be applied in the improvement of international migration process. In
this case, the sociological meaning of (5) is to maximize the social
benefits by choosing an appropriate distribution density of the
population.
\\

Indeed, many mathematical tools have been developed for the
infinite-dimensional linear programming, such as
Monge-Kantorovich-Rubinstein-Wasserstein matrices \cite{R,V}, etc.
In this paper, we investigate the analytic approximating probability
density through {\it canonical duality method} introduced by David
Y. Gao and G. Strang \cite{G1,G2,G3}. This theory was originally
proposed to find minimizers for a non-convex strain energy
functional with a double-well potential. During the last few years,
considerable effort has been taken to illustrate these non-convex
problems from the theoretical viewpoint. Through applying this
method, they characterized the local energy extrema and the global
energy minimizer for both hard device and soft device and finally
obtained the analytical solutions
\cite{G4,G5,G6}.\\

Inspired by the survey paper \cite{Evans3}, we propose an
approximation approach of nonlinear differential equation by
introducing a sequence of approximation problems for the primal
problem $(\mathcal {P})$, namely,
\begin{equation}(\mathcal{P}^{(\varepsilon)}):
\displaystyle\min_{w_\varepsilon}\Big\{I^{(\varepsilon)}[w_\varepsilon]:=\int_{\Omega}
L^{(\varepsilon)}(\nabla w_{\varepsilon},
w_\varepsilon,y)dy:=\int_{\Omega} \Big(H^{(\varepsilon)}(\nabla
w_{\varepsilon})-w_\varepsilon |y|^p\Big)dy\Big\},
\end{equation}
where $H^{(\varepsilon)}:\mathbb{R}^n\to\mathbb{R}^+$ is defined as
\[
H^{(\varepsilon)}(\gamma):=\varepsilon{\rm
e}^{(|\gamma|^2-\alpha^2)/(2\varepsilon)},
\]
and $w_\varepsilon$ is subject to the constraints (1)-(4). Moreover,
$$L^{(\varepsilon)}(P,z,y):\mathbb{R}^n\times\mathbb{R}\times\Omega\to
\mathbb{R}$$ satisfies the following coercivity inequality and is
convex in the variable $P$, $$ L^{(\varepsilon)}(P,z,y)\geq
p_{\varepsilon}|P|^2-q_\varepsilon,\ P\in\mathbb{R}^n,
z\in\mathbb{R}, y\in\Omega, $$ for certain constants $p_\varepsilon$
and $q_\varepsilon$. $I^{(\varepsilon)}$ is called the potential
energy functional and is weakly lower semicontinuous on
$W^{1,\infty}_0(\Omega)$. It's worth noticing that when
$|\gamma|\leq\alpha$, then
$$\displaystyle\lim_{\varepsilon\to
0^+}H^{(\varepsilon)}(\gamma)=0$$ uniformly. Consequently, once such
a sequence of functions $\{\bar{u}_\varepsilon\}_\varepsilon$
satisfying
$$I^{(\varepsilon)}[\bar{u}_\varepsilon]=\displaystyle\min_{w_\varepsilon}\Big\{I^{(\varepsilon)}[w_\varepsilon]\Big\}$$
is obtained, then it will help find an optimal probability density
which solves the primal problem $(\mathcal {P})$. The key mission is
to obtain an explicit representation of this approximation sequence
$\{\bar{u}_\varepsilon\}_\varepsilon$. Generally speaking, there are
plenty of approximating schemes, for example, one can also let
$$H^{(\varepsilon)}(\gamma):=\varepsilon(|\gamma|^2-\alpha^2)^2.$$ Then by following the
procedure of dealing with double-well potentials in \cite{G1}, one
could definitely find an optimal probability density.\\

By variational calculus, correspondingly, one derives a sequence of
Euler-Lagrange equations for $(\mathcal{P}^{(\varepsilon)})$, \beq
\begin{array}{ll}\displaystyle {\rm div}({\rm
e}^{(|\nabla
\bar{u}_{\varepsilon}|^2-\alpha^2)/(2\varepsilon)}\nabla
\bar{u}_{\varepsilon})+|y|^p=0,& \ \text{\rm in}\ U^{(\varepsilon)},
\end{array}\eneq
equipped with the Dirichlet boundary condition, where the compact
support
$$U^{(\varepsilon)}:=\text{Supp}(\bar{u}_\varepsilon)\subset\Omega$$ is connected and will
be determined in Lemma 2.9. The term ${\rm e}^{(|\nabla
u_{\varepsilon}|^2-\alpha^2)/(2\varepsilon)}$ is called the
transport density. Clearly, similar as $p$-Laplacian, ${\rm
e}^{(|\nabla u_{\varepsilon}|^2-\alpha^2)/(2\varepsilon)}$ is a
highly nonlinear function with respect to $\nabla u_{\varepsilon}$,
which is difficult to solve by the direct approach \cite{JH,Evans4}.
However, by the canonical duality theory, one is able to demonstrate
the existence and uniqueness of the solution of the Euler-Lagrange
equation, which establishes the equivalence between the global
minimizer of ($\mathcal{P}^{(\varepsilon)}$)
and the solution of Euler-Lagrange equation (7).\\

Now we introduce the main theorem in this paper.
\begin{thm}
Let $\{y:|y|=R_1\}\subset U^{(\varepsilon)}$ and $R_1>>R_2$. For any
$\varepsilon>0$, there exists a unique radially symmetric solution
$\bar{u}_\varepsilon$ satisfying the constraints (1)-(4) for the
Euler-Lagrange equation (7). At the same time, $\bar{u}_\varepsilon$
is a global minimizer for the approximation problem (6) in the
following explicit form $\bar{u}_\varepsilon(r)$(without any
confusion with respect to $\bar{u}_\varepsilon(y)$),
\[
\bar{u}_\varepsilon(r)= \left\{
\begin{array}{cl}
\displaystyle\int^{r}_{R_1}\rho^{1-n}\Big(C_\varepsilon(R_1)-G(\rho)\Big)/E_\varepsilon^{-1}\Big(\rho^{2-2n}(C_\varepsilon(R_1)-G(\rho))^2\Big)d\rho,&
r\in[p_\varepsilon^*(R_1),R_1],\\
\\
0,& \text{elsewhere}\ \text{in}\ \Omega;
\end{array}
\right.
\]
where $E_\varepsilon$ and $G$ are defined as
\[
\left\{
\begin{array}{ll}
E_\varepsilon(x):=\displaystyle x^2\ln({\rm
e}^{\alpha^2}x^{2\varepsilon}),&
x\in[{\rm e}^{-\alpha^2/(2\varepsilon)},1],\\
\\
G(r):=\displaystyle r^{n+p}/(n+p), & r\in[p_\varepsilon^*(R_1),R_1].
\end{array}
\right.
\]
$E_\varepsilon^{-1}$ stands for the inverse of $E_\varepsilon$, both
$C_\varepsilon(R_1)$ and $p^*_\varepsilon(R_1)$ are constants
depending on the radius $R_1$ and $\varepsilon$. Furthermore, by
letting $\varepsilon\to 0^+$, one can solve the optimization problem
(5) for the $p$-th moment of the real-valued variable $Y\in\Omega$.
That is to say, for the maximization problem ($\mathcal{P}$), there
exists a global probability density maximizer which satisfies the
constraints (1)-(4).
\end{thm}
\begin{rem}
We require $R_1>>R_2$ such that, for any $\varepsilon>0$,
\beq\displaystyle\int_{R_2}^{R_1}
r^{n-1}\int^{r}_{R_1}\rho^{1-n}\Big(C_\varepsilon-G(\rho)\Big)/E_\varepsilon^{-1}\Big(\rho^{2-2n}(C_\varepsilon-G(\rho))^2\Big)d\rho
dr>\Gamma(n/2)/(2\pi^{n/2}), \eneq where $\Gamma(x)$ stands for the
Gamma function, $C_\varepsilon$ given in Lemma 2.8 depends on $R_i$,
$i=1,2$.  This assumption is so important that it determines the
existence of such a probability density which satisfies the
normalized balance condition (3).
\end{rem}
\begin{rem}
On the one hand, $\{y:|y|=R_1\}\subset U^{(\varepsilon)}$ indicates
in the financial market, venture capitalists prefer to invest in
enterprises that are too risky for the standard capital markets or
bank loans to get a significant return through an eventual exit
event, such as IPO(initial public offerings) or trade sale of the
companies. On the other hand, $\{y:|y|=R_2\}\subset
U^{(\varepsilon)}$ models the reluctance of a risk-averse investor
to accept a bargain with higher risk rather than another bargain
with a more certain, but possibly lower expected payoff.
\end{rem}
\begin{rem}
For a general radially symmetric, positive and convex payoff
function $g(|Y|)$, by a similar approach as in the proof of Theorem
1.1, one is able to solve the following optimization problem
\beq\displaystyle\max_{{u}}\Big\{\mathbb{E}_{u}(g(|Y|)))
:=\int_{\Omega}g(|y|)u(y)dy\Big\},\eneq where $u$ is subject to the
constraints (1)-(4).
\end{rem}
\begin{rem}
For the pricing of options, volatility is a measure of the rate and
magnitude of the change of prices (up or down) of the underlying. If
volatility is high, the premium on the option will be relatively
high, and vice versa. This is in relation to the following
maximization of  variance of $\mathbb{R}$-valued random variable,
$$ \displaystyle\max_{{u}}\Big\{\mathbb{E}_{u}(Y-\mathbb{E}_u(Y))^2
=\mathbb{E}_u(Y^2)-(\mathbb{E}_u(Y))^2\Big\}.$$ If we require
\beq\mathbb{E}_u(Y)=\mu,\eneq then the problem is reduced to the
maximization of second moment of $\mathbb{R}$-valued random variable
with probability densities subject to (1)-(4) and (10). Following
the proof of Theorem 1.1, one is able to find a probability density
maximizer. If $\mathbb{E}_u(Y)$ keeps unknown, this nonlinear
optimization problem remains to be discussed theoretically.
\end{rem}
The rest of the paper is organized as follows. In Section 2, first,
we introduce some useful notations which will simplify the proof
considerably. Then, we apply the canonical dual transformation to
deduce a sequence of perfect dual problems
($\mathcal{P}^{(\varepsilon)}_d$) corresponding to
$(\mathcal{P}^{(\varepsilon)})$ and a pure complementary energy
principle. Next, we apply the canonical duality theory to prove
Theorem 1.1. A few remarks will conclude the discussion.
\section{Proof of Theorem 1.1: canonical duality approach}
\subsection{Useful notations}
Before proving the main result, first and foremost, we introduce
some useful
notations.\\

\begin{itemize}
\item $\theta_\varepsilon$ is the corresponding G\^{a}teaux derivative of
$H^{(\varepsilon)}$ with respect to $\nabla u_{\varepsilon}$, given
by \[\theta_\varepsilon(y)={\rm e}^{(|\nabla
u_{\varepsilon}|^2-\alpha^2)/(2\varepsilon)}\nabla
u_{\varepsilon}.\]
\item $\Phi^{(\varepsilon)}$ is a nonlinear
geometric mapping given by
$$\Phi^{(\varepsilon)}(u_\varepsilon):=(|\nabla u_{\varepsilon}|^2-\alpha^2)/(2\varepsilon).$$
For convenience's sake, denote
$$\xi_\varepsilon:=\Phi^{(\varepsilon)}(u_\varepsilon).$$ It is evident that $\xi_\varepsilon$ belongs to the
function space $\mathscr{U}$ given by
$$\mathscr{U}:=
\Big\{\phi\ \Big|\ \phi\leq 0\Big\}.$$
\item $\Psi^{(\varepsilon)}$ is a canonical energy
defined as $$\Psi^{(\varepsilon)}(\xi_\varepsilon):=\varepsilon{\rm
e}^{\xi_\varepsilon},$$ which is a convex function with respect to
$\xi_\varepsilon$.
\item $\zeta_\varepsilon$ is the corresponding G\^{a}teaux derivative of
$\Psi^{(\varepsilon)}$ with respect to $\xi_\varepsilon$, given by
$$\zeta_\varepsilon=\varepsilon{\rm e}^{\xi_\varepsilon},$$ which is invertible with respect to
$\xi_\varepsilon$ and belongs to the function space
$\mathscr{V}^{(\varepsilon)}$,
$$\mathscr{V}^{(\varepsilon)}:=\Big\{\phi\ \Big|\ 0<\phi\leq \varepsilon\Big\}.$$
\item
$\Psi^{(\varepsilon)}_\ast$ is defined as
\[
\Psi^{(\varepsilon)}_\ast(\zeta_\varepsilon):=\xi_\varepsilon\zeta_\varepsilon-\Psi^{(\varepsilon)}(\xi_\varepsilon)=\zeta_\varepsilon(\ln(\zeta_\varepsilon/\varepsilon)-1).
\]
\item $\lambda_\varepsilon$ is defined as $$\lambda_\varepsilon:=\zeta_\varepsilon/\varepsilon,$$ and belongs
to the function space $\mathscr{V}$,
$$\mathscr{V}:=\Big\{\phi\ \Big|\ 0<\phi\leq 1\Big\}.$$
\end{itemize}
\subsection{Canonical duality techniques}
\begin{defi}
By Legendre transformation, one defines a Gao-Strang total
complementary energy functional $\Xi^{(\varepsilon)}$,
\[
\Xi^{(\varepsilon)}(u_\varepsilon,\zeta_\varepsilon):=\displaystyle\int_{U^{(\varepsilon)}}\Big\{\Phi^{(\varepsilon)}(u_\varepsilon)\zeta_\varepsilon-\Psi^{(\varepsilon)}_\ast(\zeta_\varepsilon)
-|y|^pu_\varepsilon\Big\}dy.
\]
\end{defi}
Next, we introduce an important {\it criticality criterium} for
$\Xi^{(\varepsilon)}$.
\begin{defi}
$(\bar{u}_\varepsilon, \bar{\zeta}_\varepsilon)$ is called a
critical pair of $\Xi^{(\varepsilon)}$ if and only if \beq
D_{u_\varepsilon}\Xi^{(\varepsilon)}(\bar{u}_\varepsilon,\bar{\zeta}_\varepsilon)=0,
\eneq and \beq
D_{\zeta_\varepsilon}\Xi^{(\varepsilon)}(\bar{u}_\varepsilon,\bar{\zeta}_\varepsilon)=0,
\eneq where $D_{u_\varepsilon}, D_{\zeta_\varepsilon}$ denote the
partial G\^ateaux derivatives of $\Xi^{(\varepsilon)}$,
respectively.
\end{defi} Indeed, by variational calculus, we have the following
observation from (11) and (12).
\begin{lem}
On the one hand, for any fixed
$\zeta_\varepsilon\in\mathscr{V}^{(\varepsilon)}$, $(11)$ is
equivalent to the equilibrium equation
\[
\begin{array}{ll}\displaystyle {\rm div}(\lambda_\varepsilon \nabla \bar{u}_{\varepsilon})+|y|^p=0,& \
\text{\rm in}\ U^{(\varepsilon)}.\end{array}
\]
On the other hand, for any fixed $u_\varepsilon$ satisfying (1)-(4),
(12) is consistent with the constructive law
\[
\Phi^{(\varepsilon)}(u_\varepsilon)=D_{\zeta_\varepsilon}\Psi^{(\varepsilon)}_\ast(\bar{\zeta}_\varepsilon).
\]
\end{lem}
Lemma 2.3 indicates that $\bar{u}_\varepsilon$ from the critical
pair $(\bar{u}_\varepsilon,\bar{\zeta}_\varepsilon)$ solves the
Euler-Lagrange equation (7).
\begin{defi}
From Definition 2.1, one defines the Gao-Strang pure complementary
energy $I^{(\varepsilon)}_d$ in the form
\[
I^{(\varepsilon)}_d[\zeta_\varepsilon]:=\Xi^{(\varepsilon)}(\bar{u}_\varepsilon,\zeta_\varepsilon),
\]
where $\bar{u}_\varepsilon$ solves the Euler-Lagrange equation (7).
\end{defi}
As a matter of fact, another representation of the pure energy
$I^{(\varepsilon)}_d$, given by the following lemma, is much more
useful for our purpose.
\begin{lem} The
pure complementary energy functional $I^{(\varepsilon)}_d$ can be
rewritten as
\[
I^{(\varepsilon)}_d[\zeta_\varepsilon]=-1/2\int_{U^{(\varepsilon)}}\Big\{{\varepsilon|\theta_\varepsilon|^2/\zeta_\varepsilon}+\alpha^2\zeta_\varepsilon/\varepsilon+2\zeta_\varepsilon(\ln(\zeta_\varepsilon/\varepsilon)-1)\Big\}dy,
\]
where $\theta_\varepsilon$ satisfies \beq {\rm
div}\theta_{\varepsilon}+|y|^p=0\ \text{in}\ U^{(\varepsilon)},
\eneq equipped with a hidden boundary condition.
\end{lem}
\begin{proof}
Through integrating by parts, one has
\[
\begin{array}{lll}
I^{(\varepsilon)}_d[\zeta_\varepsilon]&=&\displaystyle-\underbrace{\int_{U^{(\varepsilon)}}\Big\{{\rm
div}(\zeta_\varepsilon
\nabla \bar{u}_{\varepsilon}/\varepsilon)+|y|^p\Big\}\bar{u}_\varepsilon dy}_{(I)}\\
\\
&&-\underbrace{1/2\int_{U^{(\varepsilon)}}\Big\{\zeta_\varepsilon|\nabla\bar{u}_{\varepsilon}|^2/\varepsilon+\alpha^2\zeta_\varepsilon/\varepsilon+2\zeta_\varepsilon(\ln(\zeta_\varepsilon/\varepsilon)-1)\Big\}dy.}_{(II)}\\
\\
\end{array}
\]
Since $\bar{u}_\varepsilon$ solves the Euler-Lagrange equation (7),
then the first part $(I)$ disappears. Keeping in mind the definition
of $\theta_\varepsilon$ and $\zeta_\varepsilon$, one reaches the
conclusion.
\end{proof}
With the above discussion, next, we establish a sequence of dual
variational problems corresponding to the approximation problems
($\mathcal{P}^{(\varepsilon)}$).
\begin{equation}
(\mathcal{P}_d^{(\varepsilon)}):\displaystyle\max_{\zeta_\varepsilon\in\mathscr{V}^{(\varepsilon)}}\Big\{I^{(\varepsilon)}_d[\zeta_\varepsilon]=-1/2\int_{U^{(\varepsilon)}}\Big\{{\varepsilon|\theta_\varepsilon|^2/\zeta_\varepsilon}+\alpha^2\zeta_\varepsilon/\varepsilon+2\zeta_\varepsilon(\ln(\zeta_\varepsilon/\varepsilon)-1)\Big\}dy\Big\}.
\end{equation}
Indeed, by calculating the G\^{a}teaux derivative of
$I_d^{(\varepsilon)}$ with respect to $\zeta_\varepsilon$, one has
\begin{lem} The variation of $I_d^{(\varepsilon)}$ with respect to $\zeta_\varepsilon$
leads to the dual algebraic equation (DAE), namely, \beq
|\theta_\varepsilon|^2={\bar{\zeta}_\varepsilon}^2(2\ln(\bar{\zeta}_\varepsilon/\varepsilon)+\alpha^2/\varepsilon)/\varepsilon,
\eneq where $\bar{\zeta}_\varepsilon$ is from the critical pair
$(\bar{u}_\varepsilon,\bar{\zeta}_\varepsilon)$.
\end{lem}
Taking into account the notation of $\lambda_\varepsilon$,  the
identity (15) can be rewritten as \beq
|\theta_\varepsilon|^2=E_\varepsilon(\lambda_\varepsilon)={\lambda}_\varepsilon^2\ln({\rm
e}^{\alpha^2}{\lambda}_\varepsilon^{2\varepsilon}). \eneq It is
evident $E_\varepsilon$ is monotonously increasing with respect to
$\lambda_\varepsilon\in[{\rm e}^{-\alpha^2/(2\varepsilon)},1]$. In
effect, $|\theta_\varepsilon|^2$ has the following asymptotic
expansion by using Taylor's expansion formula for
$\ln\lambda_\varepsilon$ at the point $\lambda_\varepsilon=1$.
\begin{lem}
If $\varepsilon>0$ is sufficiently small, then
$|\theta_\varepsilon|^2$ has the asymptotic expansion in the form of
$$|\theta_\varepsilon|^2=(\alpha^2-2\varepsilon)\lambda_\varepsilon^2+2\varepsilon\lambda_\varepsilon^3+R_\varepsilon(\lambda_\varepsilon),$$
where the remainder term
\[
|R_\varepsilon(\lambda_\varepsilon)|\leq \varepsilon
\] uniformly for any $\lambda_\varepsilon\in[{\rm e}^{-\alpha^2/(2\varepsilon)},1]$.
\end{lem}
\subsection{Proof of Theorem 1.1}
From the above discussion, one deduces that, once
$\theta_\varepsilon$ is given, then the analytic radially symmetric
solution of the Euler-Lagrange equation (7) can be represented as
\beq
\bar{u}_\varepsilon(y)=\displaystyle\int^{y}_{y_0}\eta_\varepsilon(t)dt,
\eneq where $y\in U^{(\varepsilon)}, y_0\in\partial
U^{(\varepsilon)}$,
$\eta_\varepsilon=(\eta^{(1)}_\varepsilon,\eta^{(2)}_\varepsilon,\cdots,\eta^{(n)}_\varepsilon):=\theta_\varepsilon/\lambda_\varepsilon$,
which satisfies the condition for path-independent integrals,
namely,
\[
\partial_{x_i}\eta_\varepsilon^{(j)}-\partial_{x_j}\eta_{\varepsilon}^{(i)}=0,\
\ i,j=1,\cdots,n.
\]
In the following, one is to determine the connected compact support
$U^{(\varepsilon)}$. Now we prove some useful lemmas as
prerequisites.
\begin{lem}
For any $\varepsilon>0$ and any $R\in(R_2,R_1)$, there exists a
unique smooth, radially symmetric solution $\bar{u}_{\varepsilon}$
of the Euler-Lagrange equation (7) with Dirichlet boundary in the
form of (17) in $\mathbb{B}(O,R_1)\backslash \mathbb{B}(O,R)$.
\end{lem}
\begin{proof}
Actually, in $\mathbb{B}(O,R_1)\backslash \mathbb{B}(O,R)$, a
radially symmetric solution for the Euler-Lagrange equation (13) is
of the form
\[
\theta_\varepsilon=F_\varepsilon(r)(y_1,\cdots,y_n)=F_\varepsilon\Big(\sqrt{\sum_{i=1}^ny_i^2}\Big)(y_1,\cdots,y_n),
\]
where
\[
F_\varepsilon(r)=C_\varepsilon/r^n-r^{p}/(p+n)
\]
is a general solution of the nonhomogeneous linear differential
equation
\[
rF_\varepsilon'(r)+nF_\varepsilon(r)=-r^{p},\ \ \ r\in(R,R_1),
\]
where $C_\varepsilon$ is to be determined later. From the identity
(16), one sees that there exists a unique $C^\infty$ function
$\lambda_\varepsilon\in[{\rm e}^{-\alpha^2/(2\varepsilon)},1]$ once
$C_\varepsilon$ is given. By paying attention to the Dirichlet
boundary condition, one has the radially symmetric solution
$\bar{u}_\varepsilon$ in the following form,
\[
\bar{u}_\varepsilon(r)=\int_{R_1}^r\Big(C_\varepsilon-G(\rho)\Big)/\Big(\rho^{n-1}\lambda_\varepsilon(\rho,C_\varepsilon)\Big)d\rho,
\ \ \ \ r\in[R,R_1].
\]
Recall that
\[
\bar{u}_\varepsilon(R)=\int^{G^{-1}(C_\varepsilon)}_{R_1}\eta_\varepsilon(\rho,C_\varepsilon)d\rho+\int^{R}_{G^{-1}(C_\varepsilon)}\eta_\varepsilon(\rho,C_\varepsilon)d\rho=0,
\]
and one can determine the constant $C_\varepsilon\in(G(R),G(R_1))$
uniquely. Indeed, let
\[
\mu_\varepsilon(\rho,s):=(s-G(\rho))/(\rho^{n-1}\lambda_\varepsilon(\rho,s))
\]
and
\[
M_{\varepsilon}(s):=\int^{G^{-1}(s)}_{R_1}\mu_\varepsilon(\rho,s)d\rho+\int^{R}_{G^{-1}(s)}\mu_\varepsilon(\rho,s)d\rho,
\]
where $\lambda_\varepsilon(\rho,s)$ is from (16). As a matter of
fact, $M_\varepsilon$ is strictly increasing with respect to
$s\in(G(R),G(R_1))$, which leads to
\[
C_\varepsilon=M_{\varepsilon}^{-1}(0).
\]
As a result, $C_\varepsilon$ depends on $R$ and $R_1$. In addition,
the contradiction method shows that $C_\varepsilon$ is strictly
increasing with respect to $R\in(R_2,R_1)$.
\end{proof}
With the above lemma, one is able to determine the connected compact
support $U^{(\varepsilon)}$.
\begin{lem}
Let $\{y:|y|=R_1\}\subset U^{(\varepsilon)}$ and $R_1>>R_2$. For any
$\varepsilon>0$, there exists a unique $p_\varepsilon^*(R_1)$ such
that
\[
U^{(\varepsilon)}=\overline{\mathbb{B}(O,R_1)\backslash
\mathbb{B}(O,p_\varepsilon^*(R_1))}\] and $\bar{u}_{\varepsilon}$
satisfies the normalized balance condition (3).
\end{lem}
\begin{proof}
Let $\text{\rm Supp}(\bar{u}_{\varepsilon})=[s,R_1]$ and define a
function $$\Pi:(R_2,R_1)\to\mathbb{R}^+$$ as follows,
\[\Pi(s):=2\pi^{n/2}/\Gamma(n/2)\displaystyle\int_s^{R_1} r^{n-1}\int_{R_1}^r\Big(C_\varepsilon(s)-G(\rho)\Big)/\Big(\rho^{n-1}\lambda_{\varepsilon}(\rho,C_\varepsilon(s))\Big)d\rho dr.\]
Indeed, since $C_\varepsilon$ is strictly increasing with respect to
$s\in(R_2,R_1)$, consequently, it is easy to check that $\Pi$ is a
strictly decreasing function with respect to $s\in(R_2,R_1)$. The
conclusion follows immediately when we keep in mind the assumption
(8).
\end{proof}
In the following, we verify that $\bar{u}_\varepsilon$ is exactly a
global minimizer for ($\mathcal{P}^{(\varepsilon)}$) and
$\bar{\zeta}_\varepsilon$ is a global maximizer for
($\mathcal{P}^{(\varepsilon)}_d$).
\begin{lem}(Canonical Duality Theory)
Let $\{y:|y|=R_1\}\subset U^{(\varepsilon)}$ and $R_1>>R_2$, where
$U^{(\varepsilon)}$ is determined in Lemma 2.9. For any
$\varepsilon>0$, $\bar{u}_\varepsilon$ from Lemma 2.8 is a global
minimizer for the approximation problem
($\mathcal{P}^{(\varepsilon)}$). And the corresponding
$\bar{\zeta}_\varepsilon$ is a global maximizer for the dual problem
($\mathcal{P}_d^{(\varepsilon)}$). Moreover, the following duality
identity holds, \beq
I^{(\varepsilon)}[\bar{u}_\varepsilon]=\displaystyle\min_{u_\varepsilon}I^{(\varepsilon)}[u_\varepsilon]=\Xi^{(\varepsilon)}(\bar{u}_\varepsilon,\bar{\zeta}_\varepsilon)=\displaystyle\max_{\zeta_\varepsilon}I_d^{(\varepsilon)}[\zeta_\varepsilon]=I_d^{(\varepsilon)}[\bar{\zeta}_\varepsilon],
\eneq where $u_\varepsilon$ is subject to the constraints (1)-(4)
and $\zeta_\varepsilon\in\mathscr{V}^{(\varepsilon)}$.
\end{lem}
Lemma 2.10 demonstrates that the maximization of the pure
complementary energy functional $I_d^{(\varepsilon)}$ is perfectly
dual to the minimization of the potential energy functional
$I^{(\varepsilon)}$. In effect, the identity (18) indicates there is
no duality gap between them.
\begin{proof}
On the one hand, for any function $\phi\in
W_0^{1,\infty}(U^{(\varepsilon)})$, the second variational form
$\delta_\phi^2I^{(\varepsilon)}$ is equal to\beq
\int_{U^{(\varepsilon)}}{\rm
e}^{(|\nabla\bar{u}_{\varepsilon}|^2-\alpha^2)/(2\varepsilon)}\Big\{(\nabla\bar{u}_{\varepsilon}\cdot
\nabla\phi)^2/\varepsilon+|\nabla\phi|^2\Big\}dx.\eneq On the other
hand, for any function $\psi\in\mathscr{V}^{(\varepsilon)}$, the
second variational form $\delta_\psi^2I_d^{(\varepsilon)}$ is equal
to
\beq-\int_{U^{(\varepsilon)}}\Big\{\varepsilon|\theta_\varepsilon|^2\psi^2/\bar{\zeta}_\varepsilon^3+\psi^2/\bar{\zeta}_\varepsilon\Big\}dx.
\eneq From (19) and (20), one deduces immediately that
\[
\delta^2_\phi I^{(\varepsilon)}[\bar{u}_\varepsilon]\geq0,\ \
\delta_\psi^2I_d^{(\varepsilon)}[\bar{\zeta}_\varepsilon]\leq0.
\]
\end{proof}
In the final analysis, we discuss the convergence of the sequence
$\{\bar{u}_{\varepsilon}\}_{\varepsilon}$ when $\varepsilon\to 0^+$.
According to Rellich-Kondrachov Compactness Theorem, since
$$\displaystyle\sup_{\varepsilon}|\bar{u}_\varepsilon|\leq \alpha R_1$$ and
$$\displaystyle\sup_{\varepsilon}|\nabla\bar{u}_{\varepsilon}|\leq\alpha,$$ then, there exists a
subsequence $\{\bar{u}_{\varepsilon_k}\}_{\varepsilon_k}$ and $f\in
W_0^{1,\infty}(\Omega)\cap C(\overline{\Omega})$ such that
\beq\bar{u}_{\varepsilon_k}\rightarrow f\ (k\to\infty)\ \text{in}\
L^\infty(\Omega),\eneq \beq\nabla\bar{u}_{\varepsilon_k}\
\overrightharpoon{*}\ \nabla f\ (k\to\infty)\ \text{weakly\ $\ast$\
in}\ L^\infty(\Omega).\eneq It remains to check that $f$ satisfies
(1)-(4). From (21), one knows \beq\bar{u}_{\varepsilon_k}\rightarrow
f\ (k\to\infty)\ \text{a.e.\ in}\ \Omega.\eneq According to
Lebesgue's dominated convergence theorem,
\[\displaystyle\int_{\Omega}f(y)dy=\lim_{k\to\infty}\int_{\Omega}\bar{u}_{\varepsilon_k}(y)dy=1.\]
From (22), one has
\[\|\nabla f\|_{L^\infty(\Omega)}\leq\displaystyle\liminf_{k\to\infty}\|\nabla\bar{u}_{\varepsilon_k}\|_{L^\infty(\Omega)}\leq\sup_{k\to\infty}\|\nabla\bar{u}_{\varepsilon_k}\|_{L^\infty(\Omega)}\leq\alpha.\]
Consequently, one reaches the conclusion of Theorem 1.1 by
summarizing the above discussion.\\
\begin{rem}
In this paper, we mainly focus on the construction of maximizers
through the approximation procedure. Rather than the
infinite-dimensional linear programming, we provide another
viewpoint and give the explicit representation of the approximating
probability densities by applying the canonical duality method.
Furthermore, the canonical duality method proves to be useful and
can also be applied in the discussion of solutions for the
$p$-Laplacian problems etc.\cite{LU2}
\end{rem}
{\bf Acknowledgement}:  The main results in this paper were obtained
during a research collaboration at the Federation University
Australia in July, 2016. both authors wish to thank Professor David
Y. Gao for his hospitality and financial support. This project is
partially supported by US Air Force Office of Scientific Research
(AFOSR FA9550-10-1-0487), Natural Science Foundation of Jiangsu
Province (BK 20130598), National Natural Science Foundation of China
(NSFC 71273048, 71473036, 11471072), the Scientific Research
Foundation for the Returned Overseas Chinese Scholars, Fundamental
Research Funds for the Central Universities on the Field Research of
Commercialization of Marriage between China and Vietnam (No.
2014B15214). This work is also supported by Open Research Fund
Program of Jiangsu Key Laboratory of Engineering Mechanics,
Southeast University (LEM16B06). In particular, the authors also
express their deep gratitude to the referees for their careful
reading and useful remarks.

\end{document}